\documentclass[11pt, a4paper, notitlepage]{article}

\usepackage{amsmath, latexsym, amsthm, amsfonts} 
\usepackage{graphicx} 
\usepackage{epsfig}
\usepackage{varioref}
\usepackage{natbib} 

\bibpunct{(}{)}{;}{a}{}{,} 

\theoremstyle{plain}
\newtheorem{theorem}{Theorem}
\newtheorem{lemma}{Lemma}

\newcommand{\infint}{\int_{-\infty}^{\infty}}
\def\convd{\stackrel{\cal D}{\rightarrow}}

\def\ex{{\rm E\,}}
\def\var{\mathop{\rm Var}\nolimits}

\begin{document}
\title{Asymptotic normality of the deconvolution kernel density estimator under the vanishing error variance}
\author{Bert van Es\\
{\normalsize Korteweg-de Vries Institute for Mathematics}\\
{\normalsize Universiteit van Amsterdam}\\
{\normalsize Plantage Muidergracht 24}\\
{\normalsize 1018 TV Amsterdam}\\
{\normalsize The Netherlands}\\
{\normalsize a.j.vanes@uva.nl}\\
{}\\
{\normalsize Shota Gugushvili}\\
{\normalsize EURANDOM}\\
{\normalsize Technische Universiteit Eindhoven}\\
{\normalsize P.O.\ Box 513}\\
{\normalsize 5600 MB Eindhoven}\\
{\normalsize The Netherlands}\\
{\normalsize gugushvili@eurandom.tue.nl}\\
{\normalsize Phone: +31 40 2478113}\\
{\normalsize Fax: +31 40 2478190}} \maketitle
\begin{abstract}
Let $X_1,\ldots,X_n$ be i.i.d.\ observations, where
$X_i=Y_i+\sigma_n Z_i$ and the $Y$'s and $Z$'s are independent.
Assume that the $Y$'s are unobservable and that they have the
density $f$ and also that the $Z$'s have a known density $k.$
Furthermore, let $\sigma_n$ depend on $n$ and let
$\sigma_n\rightarrow 0$ as $n\rightarrow\infty.$ We consider the
deconvolution problem, i.e.\ the problem of estimation of the
density $f$ based on the sample $X_1,\ldots,X_n.$ A popular
estimator of $f$ in this setting is the deconvolution kernel
density estimator. We derive its asymptotic normality under two
different assumptions on the relation between the sequence
$\sigma_n$ and the sequence of bandwidths $h_n.$ We also consider several simulation examples which illustrate different types of asymptotics corresponding to the derived theoretical results and which show that there exist situations where models with $\sigma_n\rightarrow 0$ have to be preferred to the models with fixed $\sigma.$
{\sl Running title:} Deconvolution kernel density estimator
\end{abstract}
\newpage


\section{Introduction}

The classical deconvolution problem consists of estimation of the
density $f$ of a random variable $Y$ based on the i.i.d.\ copies
$Y_1,\ldots,Y_n$ of $Y,$ which are corrupted by an additive
measurement error. More precisely, let $X_1,\ldots,X_n$ be i.i.d.\
observations, where $X_i=Y_i+Z_i$ and the $Y$'s and $Z$'s are
independent. Assume that the $Y$'s are unobservable and that they
have the density $f$ and also that the $Z$'s have a known density
$k.$ Such a model of measurements contaminated by an additive
measurement error has numerous applications in practice and arises
in a variety of fields, see for instance \citet{carroll3}. Notice
that the $X$'s have a density $g$ which is equal to the
convolution of $f$ and $k.$ The deconvolution problem consists in
estimation of the density $f$ based on the sample
$X_1,\ldots,X_n.$

A popular estimator of $f$ is the deconvolution kernel density
estimator, which was proposed in \citet{carroll1} and
\citet{carroll2}, see also pp.\ 231--233 in \citet{wasserman} for
an introduction. Additional recent references can be found e.g.\
in \citet{gugu2}. Let $w$ be a kernel and $h_n>0$ a bandwidth. The
deconvolution kernel density estimator $f_{nh_n}$ is constructed
as
\begin{equation}
\label{fnh} f_{nh_n}(x)=\frac{1}{2\pi}\infint
e^{-itx}\frac{\phi_w(h_nt)\phi_{emp}(t)}{\phi_k(t)}dt=\frac{1}{n}\sum_{j=1}^n
\frac{1}{h_n}w_{h_n}\left(\frac{x-X_j}{h_n}\right),
\end{equation}
where $\phi_{emp}$ denotes the empirical characteristic function, i.e.\ $\phi_{emp}(t)=n^{-1}\sum_{j=1}^n \exp(itX_j),$
$\phi_w$ and $\phi_k$ are Fourier transforms of functions $w$ and $k,$ respectively, and
\begin{equation*}
w_{h_n}(x)=\frac{1}{2\pi}\infint e^{-itx}\frac{\phi_w(t)}{\phi_k(t/h_n)}dt.
\end{equation*}
Depending on the rate of decay of the characteristic function
$\phi_k$ at plus and minus infinity, deconvolution problems are
usually divided into two groups, ordinary smooth deconvolution
problems and supersmooth deconvolution problems. In the first case
it is assumed that $\phi_k$ decays to zero at plus and minus
infinity algebraically (an example of such $k$ is the Laplace
density) and in the second case the decay is essentially
exponential (in this case $k$ can be e.g.\ a standard normal
density). In general, the faster $\phi_k$ decays at plus and minus
infinity (and consequently smoother the density $k$ is), the more
difficult the deconvolution problem becomes, see e.g.\
\citet{fan1}. The usual smoothness condition imposed on the target
density $f$ is that it belongs to the class ${\mathcal
C}_{\alpha,L}=\{f:|f^{(\ell)}(x)-f^{(\ell)}(x+t)|\leq
L|t|^{\alpha-\ell}\text{ for all }x\text{ and }t\},$ where
$\alpha>0,$ $\ell=\left\lfloor \alpha \right\rfloor$ (the integer
part of $\alpha$) and $L>0$ are known constants, cf.\
\citet{fan1}. Then, if $k$ is ordinary smooth of order $\beta$
(see e.g.\ Assumption C (ii) below for a definition), the optimal
rate of convergence for the estimator
$f_{nh_n}(x)$ with the mean square error used as the performance criterion is $n^{-\alpha/(2\alpha+2\beta+1)},$ while if $k$
is supersmooth of order $\lambda$ (see Assumption B (ii)), the
optimal rate of convergence is $(\log n)^{-\alpha/\lambda},$ see
\citet{fan1}. The latter convergence rate is rather slow and it
suggests that the deconvolution problem is not practically
feasible in the supersmooth case, since it seems samples of very
large size are required to obtain reasonable estimates. Hence at
first sight it appears that the nonparametric deconvolution with
e.g.\ the Gaussian error distribution (a popular choice in
practice) cannot lead to meaningful results for moderate sample
sizes and is practically irrelevant. However, it was demonstrated
by exact $\operatorname{MISE}$ (mean integrated square error)
computations in \citet{wand} that, despite the slow convergence
rate in the supersmooth case, the deconvolution kernel density
estimator performs well for reasonable sample sizes, if the noise
level measured by the noise-to-signal ratio
$\operatorname{NSR}=\var[Z](\var[Y])^{-1}100\%,$ cf.\
\citet{wand}, is not too high. Clearly, an `ideal case' in a
deconvolution problem would be that not only the sample size $n$
is large, but also that the error term variance is small. This
leads one to an idealised model $X=Y+\sigma_n Z,$ where now
$\var[Z]=1$ and $\sigma_n$ depends on $n$ and tends to zero as
$n\rightarrow \infty.$ The idea to consider $\sigma_n\rightarrow
0$ was already proposed in \citet{fan3} and was further developed
in \citet{delaigle0}. We refer to these works for additional
motivation. These papers deal mainly with the mean integrated
square error of the estimator of $f.$ Here we will study its
asymptotic normality. Asymptotic normality of the deconvolution
kernel density estimator in the deconvolution problem with fixed
error term variance was derived in \citet{fan2} and
\citet{vanes2,vanes1}. For a practical situation where $\sigma_n\rightarrow 0$ can arise, see e.g.\ Section 4.2 of \citet{delaigle0}, where an example of measurement of sucrase in intestinal tissues is considered and inference is drawn on the density of the sucrase content. Sucrase is a name of several enzymes that catalyse the hydrolisis of sucrose to fructose and glucose.

It trivially follows from \eqref{fnh} that the deconvolution
kernel density estimator for the model that we consider, i.e.\
$X_i=Y_i+\sigma_n Z_i$ with $\sigma_n\rightarrow 0$ as $n\rightarrow\infty,$ is defined as
\begin{equation}
\label{fnh2} f_{nh_n}(x)=\frac{1}{2\pi}\infint
e^{-itx}\frac{\phi_w(h_nt)\phi_{emp}(t)}{\phi_k(\sigma_n
t)}dt=\frac{1}{n}\sum_{j=1}^n \frac{1}{h_n}
w_{r_n}\left(\frac{x-X_j}{h_n}\right),
\end{equation}
where
\begin{equation}
\label{wh} w_{r_n}(x)=\frac{1}{2\pi}\infint
e^{-itx}\frac{\phi_w(t)}{\phi_k(r_nt)}dt,
\end{equation}
$r_n=\sigma_n/ h_n$ and $\phi_k$ now denotes the characteristic
function of the random variable $Z$ with a density $k.$ We will
also use $\rho_n=r_n^{-1}=h_n/\sigma_n$ and in this case we will
denote the function $w_{r_n}$ by $w_{\rho_n}.$ Observe that if $w$
is symmetric, \eqref{fnh2} will be real-valued.

To get a consistent estimator, we need to control the bandwidth
$h_n.$ The usual condition to get consistency in kernel density estimation is that the bandwidth $h_n$ depends on $n$ and is such that $h_n\rightarrow 0,nh_n\rightarrow\infty,$ see e.g.\ Theorem 6.27 in
\citet{wasserman}. Since in our model we assume $\sigma_n\rightarrow
0,$ additional assumptions on $h_n,$ which relate it to $\sigma_n,$
are needed. In essence we distinguish two cases:
$\sigma_n/ h_n\rightarrow r$ with $0\leq r<\infty,$ or
$\sigma_n/ h_n\rightarrow\infty.$ Conditions on the target density $f,$
the density $k$ of $Z$ and kernel $w$ will be tailored to these two
cases.

The remaining part of the paper is organised as follows: in
Section \ref{results} we will present the obtained results. Section \ref{simulations} contains several simulation examples illustrating the results from Section \ref{results}. All the proofs are given in Section \ref{proofs}.

\section{Results}
\label{results}

\subsection{The case $0\leq r<\infty$}

We first consider the case when $0\leq r<\infty.$ We will need the
following conditions on $f,$ $w,$ $k$ and $h_n.$

\bigskip

{\bf Assumption A.}

(i) The density $f$ is such that $\phi_f$ is integrable.

(ii) $\phi_k(t)\neq 0$ for all $t\in{\mathbb R}$ and $\phi_k$ has
a bounded derivative.

(iii) The kernel $w$ is symmetric, bounded and continuous.
Furthermore, $\phi_w$ has support $[-1,1],$ $\phi_w(0)=1,$ $\phi_w$ is differentiable and
$|\phi_w(t)|\leq 1.$

(iv) The bandwidth $h_n$ depends on $n$ and we have $h_n\rightarrow
0,nh_n\rightarrow\infty.$

(v) $\sigma_n\rightarrow 0$ and $r_n=\sigma_n/ h_n\rightarrow r,$ where
$0\leq r<\infty.$

\bigskip

Notice that Assumption A (i) implies that $f$ is continuous and bounded.
Assumption $\phi_k(t)\neq 0$ for all $t\in{\mathbb R}$ is standard in kernel deconvolution and is unavoidable when using the Fourier inversion approach to deconvolution.
Furthermore, a variety of kernels satisfy Assumption A (iii), see e.g.\ examples in \citet{vanes1}. Also notice that $w$ is not necessarily a density, since it may take on negative values. Observe that in Assumption A (v) we do not
exclude the case $r=0.$

The following theorem establishes asymptotic normality in this
case.
\begin{theorem}
\label{thman1} Let Assumption A hold and let the estimator
$f_{nh_n}$ be defined by \eqref{fnh2}. Then
\begin{equation}
\label{an0} \sqrt{nh_n}(f_{nh_n}(x)-\ex[f_{nh_n}(x)])\convd {\mathcal
N}\left(0\, , \, f(x)\negthinspace\infint |w_r(u)|^2du\right)
\end{equation}
as $n\rightarrow\infty.$
\end{theorem}
Notice that unlike the asymptotic normality theorem for the
deconvolution kernel density estimator in the supersmooth
deconvolution problem with fixed $\sigma,$ that was obtained in
\citet{vanes2,vanes1}, the asymptotic variance in \eqref{an0} now
depends on $f$. When $r_n=0$ for all $n,$ we recover the
asymptotic normality theorem for an ordinary kernel density
estimator, see \citet{parzen}.

\subsection{The case $r=\infty$}

We turn to the case $r=\infty.$ In this case we have to make the
distinction between the ordinary smooth and supersmooth
deconvolution problems. We first consider the supersmooth
case. We will need the following condition.

\bigskip

{\bf Assumption B.}

(i) The density $f$ is such that $\phi_f$ is integrable.

(ii) $\phi_k(t)\neq 0$ for all $t\in{\mathbb R}$ and
$\phi_k(t)\sim C|t|^{\lambda_0}\exp(-|t|^{\lambda}/\mu)$ for some
constants $\lambda>1,\mu>0$ and real constants $\lambda_0$ and
$C.$

(iii) $w$ is a bounded, symmetric and continuous function.
Furthermore, $\phi_w$ is supported on $[-1,1],$ $\phi_w(0)=1$ and
$|\phi_w(t)|\leq 1.$ Moreover,
\begin{equation*}
\phi_w(1-t)=At^{\alpha}+o(t^{\alpha})
\end{equation*}
as $t\downarrow 0,$ where $A\in\mathbb{R}$ and $\alpha\geq 0$ are some numbers.

(iv) The bandwidth $h_n$ depends on $n$ and we have $h_n\rightarrow
0,nh_n\rightarrow\infty.$

(v) $\sigma_n\rightarrow 0$ and $\sigma_n^{\lambda}/
h_n^{\lambda-1}\rightarrow\infty.$

\bigskip

Assumption B (i)-(iv) correspond to those in \citet{vanes1}. Assumption B (v) is stronger than $\sigma_n/h_n\rightarrow\infty,$ but
it is essential in the proof of Theorem \ref{thman3}. Denote
$\zeta(\rho_n)=\exp(1/(\mu \rho_n^{\lambda})).$ The following
theorem holds true.

\begin{theorem}
\label{thman3} Let Assumption B hold and let the estimator
$f_{nh_n}$ be defined by \eqref{fnh2}. Furthermore, assume that $\ex[Y^2_j]<\infty$ and $\ex[Z^2_j]<\infty.$ Then
\begin{equation}
\label{supersmooth}
\frac{\sqrt{n}\sigma_n}{\rho_n^{\lambda(1+\alpha)+\lambda_0-1}\zeta(\rho_n)}(f_{nh_n}(x)-\ex[f_{nh_n}(x)])\convd
{\mathcal N} \left(0,\frac{A^2}{2\pi^2C^2}\left(\frac{\mu}
{\lambda}\right)^{2+2\alpha}(\Gamma(\alpha+1))^2\right)
\end{equation}
as $n\rightarrow\infty.$
\end{theorem}
When $\sigma_n=1$ for all $n,$ the arguments given in the proof of this theorem are still valid, and hence we can also recover
the asymptotic normality theorem of \citet{vanes1} for the deconvolution kernel density estimator in the supersmooth deconvolution problem.

Finally, we consider the ordinary smooth case.

\bigskip

{\bf Assumption C.}

(i) The density $f$ is such that $\phi_f$ is integrable.

(ii) $\phi_k(t)\neq 0$ for all $t\in{\mathbb R}$ and
$\phi_k(t)t^{\beta}\rightarrow
C,\phi_k^{\prime}(t)t^{\beta+1}\rightarrow-\beta C$ as
$t\rightarrow\infty,$ where $\beta\geq 0$ and $C\neq 0$ are some
constants.

(iii) $\phi_w$ is symmetric and continuously differentiable.
Furthermore, $\phi_w$ is supported on $[-1,1],$ $|\phi_w(t)|\leq
1$ and $\phi_w(0)=1.$

(iv) The bandwidth $h_n$ depends on $n$ and we have $h_n\rightarrow
0,nh_n\rightarrow\infty.$

(v) $\sigma_n\rightarrow 0$ and $\sigma_n/ h_n\rightarrow\infty.$

\bigskip

For the discussion on Assumption C (i)--(iv) see \citet{fan2}.

\begin{theorem}
\label{thman2} Let Assumption C hold and let the estimator
$f_{nh_n}$ be defined by \eqref{fnh2}. Then
\begin{equation}
\label{ordinarysmooth}
\sqrt{nh_n\rho_n^{2\beta}}(f_{nh_n}(x)-\ex[f_{nh_n}(x)])\convd {\mathcal
N} \left(0,\frac{f(x)}{{2\pi C^2}}\infint
|t|^{2\beta}|\phi_w(t)|^2dt\right)
\end{equation}
as $n\rightarrow\infty.$
\end{theorem}

When $\sigma_n=1,$ we recover the asymptotic normality theorem of
\citet{fan2} for a deconvolution kernel density estimator in the
ordinary smooth deconvolution problem.

As a general conclusion, we notice that Theorems
\ref{thman1}--\ref{thman2} demonstrate that the asymptotics of
$f_{nh_n}(x)$ depend in an essential way on the relationship
between the sequences $\sigma_n$ and $h_n.$ In case
$r_n\rightarrow r<\infty,$ the asymptotics are similar to those in
the direct density estimation, while when $r=\infty,$ they
resemble those in the classical deconvolution problem.

\section{Simulation examples}
\label{simulations}

In this section we consider several simulation examples for the
supersmooth deconvolution case covered by Theorems \ref{thman1}
and \ref{thman3}. We do not pretend to produce an exhaustive
simulation study. Our examples serve as a mere illustration of the
asymptotic results from the previous section.

It follows from Theorems \ref{thman1}--\ref{thman2} that for a
fixed point $x$ and a large enough $n,$ a suitably centred and
normalised estimator $f_{nh_n}(x)$ is approximately normally
distributed with mean and standard deviation given in these three
theorems. Suppose we have fixed the sample size $n$ and the
bandwidth $h_n,$ generated a sample of size $n,$ evaluated the
estimate $f_{nh_n}(x)$ and have repeated this procedure $N$ times,
where $N$ is sufficiently large. This will give us $N$ values of
$f_{nh_n}(x).$ We then can evaluate the sample mean and the sample
standard deviation of this set of values $f_{nh_n}(x).$ Under
appropriate conditions these should be close to the ones predicted
by Theorems \ref{thman1} and \ref{thman3}. In particular, in the
setting of Theorem \ref{thman1}, the mean $M$ and the standard
deviation $SD$ must be approximately given by
\begin{equation}
\label{msdthm1}
M=f\ast w_{h_n}(x), \quad SD=\frac{1}{\sqrt{nh_n}}f(x)\infint |w_{\sigma_n/h_n}(u)|^2du,
\end{equation}
while in the setting of Theorem \ref{thman3} they are approximately equal to
\begin{equation}
\label{msdthm3}
M=f\ast w_{h_n}(x), \quad SD=
\frac{A}{\sqrt{2}\pi C}\left(\frac{\mu}{\lambda}\right)^{1+\alpha}\Gamma(\alpha+1)\frac{\rho_n^{\lambda(1+\alpha)+\lambda_0-1}\zeta(\rho_n)}{\sqrt{n}\sigma_n}.
\end{equation}

We first concentrate on Theorem \ref{thman1}. Let $f$ and $k$ be
standard normal densities, let $n=1000$ and suppose
$\sigma_n=0.1.$ The noise level measured by the noise-to-signal
ratio is thus rather low and equals $\operatorname{NSR}=1\%.$
Suppose that a kernel $w$ is given by
\begin{equation}
\label{fankernel}
w(x)=\frac{48\cos x}{\pi x^4}\left(1-\frac{15}{x^2}\right)-\frac{144\sin x}{\pi x^5}\left(2-\frac{5}{x^2}\right).
\end{equation}
Its corresponding Fourier transform is given by
$\phi_{w}(t)=(1-t^2)^3 1_{[-1,1]}(t).$ Here $A=8$ and $\alpha=3.$
A good performance of this kernel in deconvolution context was
established in \citet{delaigle1}. Assume that the number of
replications $N=500.$ Before we proceed any further, we need to
fix the bandwidth. We opted for a theoretically optimal bandwidth,
i.e.\ the bandwidth that minimises
\begin{equation}
\label{mise}
\operatorname{MISE}[f_{nh_n}]=\ex \left[\infint (f_{nh_n}(x)-f(x))^2dx \right],
\end{equation}
the mean-squared error of the estimator $f_{nh}.$ To find this
optimal bandwidth, we considered a sequence of bandwidths
$h=0.01*k,k=1,2,\ldots,K,$ where $K$ is a large enough integer,
passed to the Fourier transforms in \eqref{mise} via Parseval's
identity, cf.\ \citet{wand}, and then used the numerical
integration. This procedure resulted in $h_n=0.1.$ For real data
the above method does not work, because \eqref{mise} depends on
the unknown $f,$ and we refer to \citet{delaigle0} for
data-dependent bandwidth selection methods. However, once again we
stress the fact that in order to reach a specific goal of these
simulation examples, the bandwidth $h_n$ must be the same for all
$N$ replications. This excludes the use of a data-dependent
procedure. To speed up the computation of the estimates, binning of observations was used, see e.g.\ \citet{silverman} and \citet{lotwick} for related ideas in kernel density estimation.

Under these assumptions we evaluated the sample means and standard
deviations of $f_{nh_n}(x)$ for $x$ from a grid on the interval
$[-3,3]$ with mesh size $\Delta=0.1.$ These then were plotted in
Figure \ref{thm1fig1msd} together with the theoretical values from
\eqref{msdthm1}. We notice that the sample means match the
theoretical values very well. This can be also explained by the fact that the bandwidth $h_n$ is quite small. The match between the sample
standard deviations and the theoretical standard deviations is
slightly less satisfactory. It also turns out that Theorem
\ref{thman3} is clearly not applicable in this case: an evaluation
of the theoretical standard deviation $SD$ in \eqref{msdthm3}
yields a very large value $3.41646,$
which grossly overestimates the sample standard deviation for any
point $x.$ The reason for this seems to be that both the sample size $n$
and the error variance $\sigma_n^2$ appear to be too small for the
setting of Theorem \ref{thman3}.
\begin{figure}[htb]
\setlength{\unitlength}{1cm}
\begin{minipage}{5.5cm}
\begin{picture}(5.5,4.0)
\epsfxsize=5.5cm\epsfysize=4cm\epsfbox{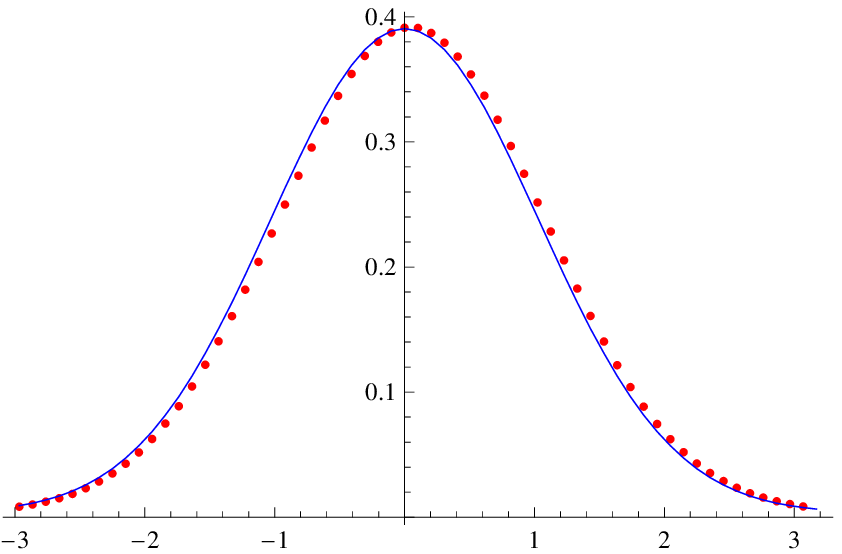}
\end{picture}
\end{minipage}
\hfill
\setlength{\unitlength}{1cm}
\begin{minipage}{5.5cm}
\begin{picture}(5.5,4.0)
\epsfxsize=5.5cm\epsfysize=4cm\epsfbox{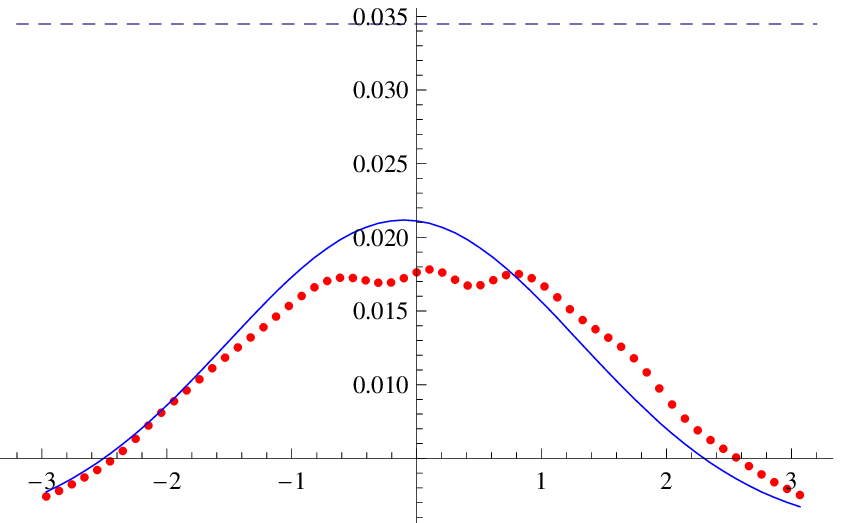}
\end{picture}
\end{minipage}
\caption{\label{thm1fig1msd} \small{The sample means and the theoretical means (left display, a dotted and a solid line, respectively) together with the sample standard deviations and the
two theoretical standard deviations corresponding to Theorems
\ref{thman1} and \ref{thman3} (right display, a dotted, a solid
and a dashed line, respectively). Here the target density $f$ and the density $k$ of a random variable $Z$ are standard normal densities, the noise variance $\sigma_n^2=0.01,$ the sample size $n=1000,$ the bandwidth $h_n=0.1$ and the kernel $w$ is given by \eqref{fankernel}. The number of replications equals $N=500.$ The integral in \eqref{asnrm2} and not its asymptotic expansion
was used to evaluate the standard deviation in Theorem
\ref{thman3}.}}
\end{figure}

At this point the following remark is in order. Reviewing the proof of Theorem \ref{thman3}, one sees that the following asymptotic equivalence is used:
\begin{equation}
\label{asnrm2}
\int_{0}^1 \phi_w(s) \exp[s^{\lambda}/(\mu h^{\lambda})]ds \sim A \Gamma (\alpha+1) \left(\frac{\mu}{\lambda}h^{\lambda}\right)^{1+\alpha} e^{1/(\mu h^{\lambda})}
\end{equation}
as $h\rightarrow0.$ This explains the shape of the normalising
constant in Theorem \ref{thman3}. However, the direct numerical
evaluation of the integral in \eqref{asnrm2} (with the same
parameters and the kernel as in our example above) shows that the
approximation in \eqref{asnrm2} is good only for very small values
of $h$ and that it is quite inaccurate for larger values of
$h,$ see a discussion in \citet{gugu1}. Obviously, one can
correct for the poor approximation of the sample standard
deviation by the theoretical standard deviation by using the
left-hand side of \eqref{asnrm2} instead of its approximation.
Nevertheless, this still leads to a very large (compared to the sample standard deviation) value of the
theoretical standard deviation for our particular example, namely $0.034477.$

In our second example we left $\sigma_n,n$ and $k$ the same as above, but as $f$ we took a mixture of two normal densities with means $-1$ and $1$ and equal variance $0.375.$ The mixing probability was taken to be equal to $0.5.$ The density $f$ is bimodal and is plotted in Figure \ref{bimodalfig}.
\begin{figure}[htb]
\centering
\includegraphics{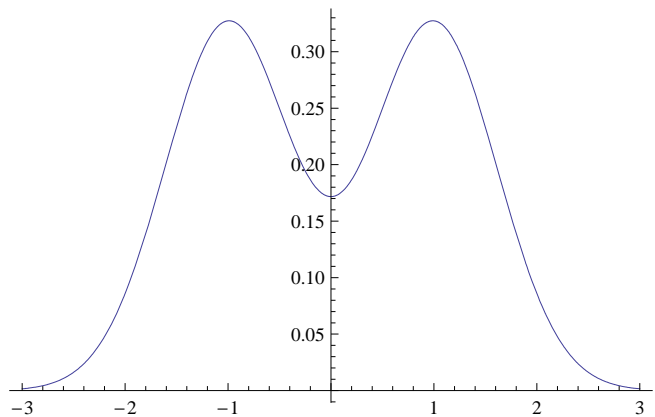}
\caption{\small{The density $f$: a mixture of two normal densities with means $-1$ and $1$ and equal variance $0.375.$ The mixing probability is taken to be equal to $0.5.$}}
\label{bimodalfig}
\end{figure}
The simulation results for this density are reported in Figure \ref{add4}. The conclusions are the same as for the first example. One can easily recognise a bimodal shape of the target density $f$ by looking at the sample standard deviation.
\begin{figure}[htb]
\setlength{\unitlength}{1cm}
\begin{minipage}{5.5cm}
\begin{picture}(5.5,4.0)
\epsfxsize=5.5cm\epsfysize=4cm\epsfbox{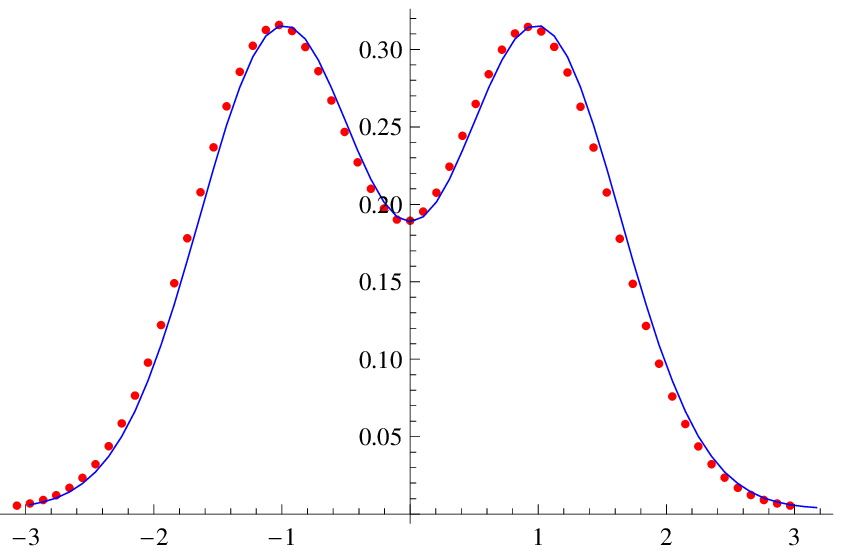}
\end{picture}
\end{minipage}
\hfill
\setlength{\unitlength}{1cm}
\begin{minipage}{5.5cm}
\begin{picture}(5.5,4.0)
\epsfxsize=5.5cm\epsfysize=4cm\epsfbox{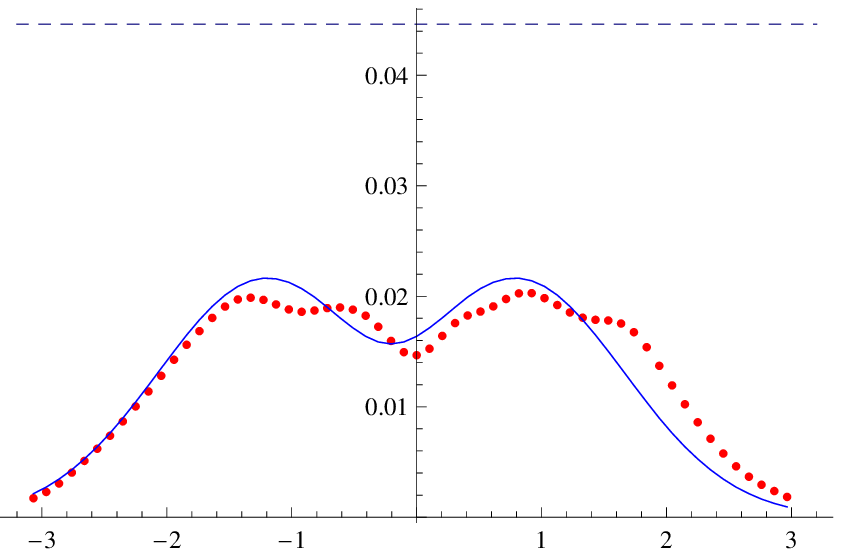}
\end{picture}
\end{minipage}
\caption{\label{add4} \small{The sample means and the
theoretical means (left display, a dotted and a solid line,
respectively) together with the sample standard deviations and the
two theoretical standard deviations corresponding to Theorems
\ref{thman1} and \ref{thman3} (right display, a dotted, a solid
and a dashed line, respectively). Here the target density $f$ is a
mixture of two normal densities with means equal to $-1$ and $1$ and the same variance $0.375,$ the mixing probability is $0.5,$ the density $k$ of a random variable $Z$ is a standard normal density, the noise variance $\sigma_n^2=0.01,$ the
sample size $n=1000,$ the bandwidth $h_n=0.08$ and the kernel
$w$ is given by \eqref{fankernel}. The number of replications
equals $N=500.$ The integral in \eqref{asnrm2} and not its
asymptotic expansion was used to evaluate the standard deviation
in Theorem \ref{thman3}.}}
\end{figure}

In our third example we again considered the standard normal density, but we increased the sample size to $n=10000.$ The results are reported in Figure \ref{thm1fig2msd}. As can be seen, the match between the sample standard deviations and the theoretical standard deviations as computed using Theorem \ref{thman1} is less satisfactory than in the previous example. The explanation lies in the fact that, even though the noise level is low when judged by itself, it is still a bit large compared to the sample size that we have in this case. Also Theorem \ref{thman3} remains unapplicable, as it still produces considerably larger values of the theoretical standard deviation compared to the sample standard deviation ($0.0166319$ after the necessary correction using \eqref{asnrm2}).
\begin{figure}[htb]
\setlength{\unitlength}{1cm}
\begin{minipage}{5.5cm}
\begin{picture}(5.5,4.0)
\epsfxsize=5.5cm\epsfysize=4cm\epsfbox{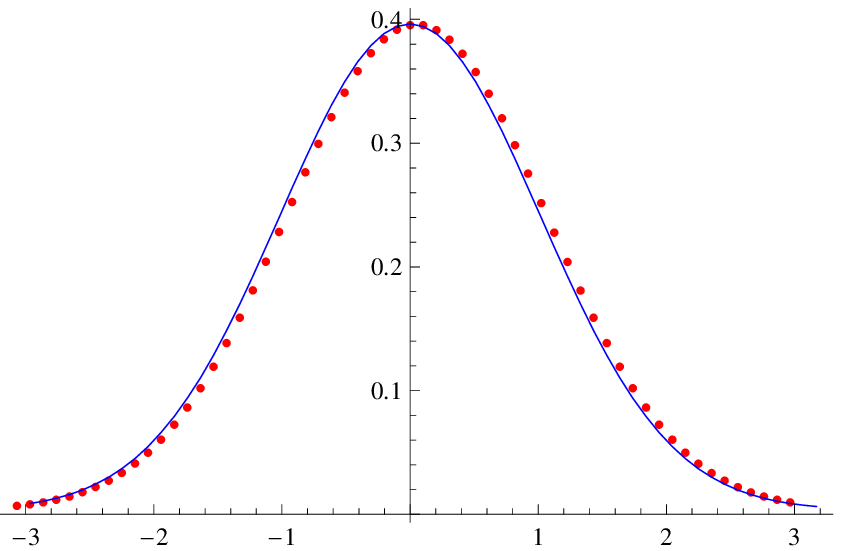}
\end{picture}
\end{minipage}
\hfill
\setlength{\unitlength}{1cm}
\begin{minipage}{5.5cm}
\begin{picture}(5.5,4.0)
\epsfxsize=5.5cm\epsfysize=4cm\epsfbox{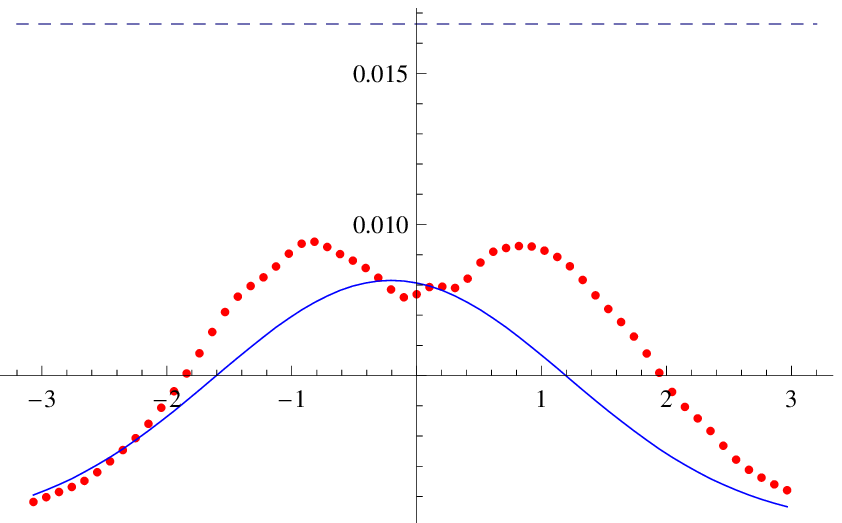}
\end{picture}
\end{minipage}
\caption{\label{thm1fig2msd} \small{The sample means and the theoretical means (left display, a dotted and a solid line, respectively) together with the sample standard deviations and the
two theoretical standard deviations corresponding to Theorems
\ref{thman1} and \ref{thman3} (right display, a dotted, a solid
and a dashed line, respectively). Here the target density $f$ and the density $k$ of a random variable $Z$ are standard normal densities, the noise variance $\sigma_n^2=0.01,$ the sample size $n=10000,$ the bandwidth $h_n=0.07$ and the kernel $w$ is given by \eqref{fankernel}. The number of replications equals $N=500.$ The integral in \eqref{asnrm2} and not its asymptotic expansion
was used to evaluate the standard deviation in Theorem
\ref{thman3}.}}
\end{figure}

In the next three examples we kept the standard normal densities $f$ and $k,$ but increased the sample size $n$ to $100000.$ The error variance $\sigma_n^2$ was consecutively taken to be $0.01,1$ and $4,$ i.e. we considered three different noise levels, $1\%,100\%$ and $400\%.$ A transition from the asymptotics described by Theorem \ref{thman1} to those described by Theorem \ref{thman3} is clearly visible in the resulting plots, see Figures \ref{add1}--\ref{add3}. Figure \ref{add1} also indicates that there exist intermediate situations not immediately covered by either of the two theorems. Notice that Figure \ref{add3} seems to confirm a general, albeit not intuitive message of Theorem \ref{thman3}, which says that the asymptotic standard deviation does not depend on a point $x,$ but only on the error density $k:$ there is a large neighbourhood around zero for which the sample standard deviation is almost constant.
\begin{figure}[htb]
\setlength{\unitlength}{1cm}
\begin{minipage}{5.5cm}
\begin{picture}(5.5,4.0)
\epsfxsize=5.5cm\epsfysize=4cm\epsfbox{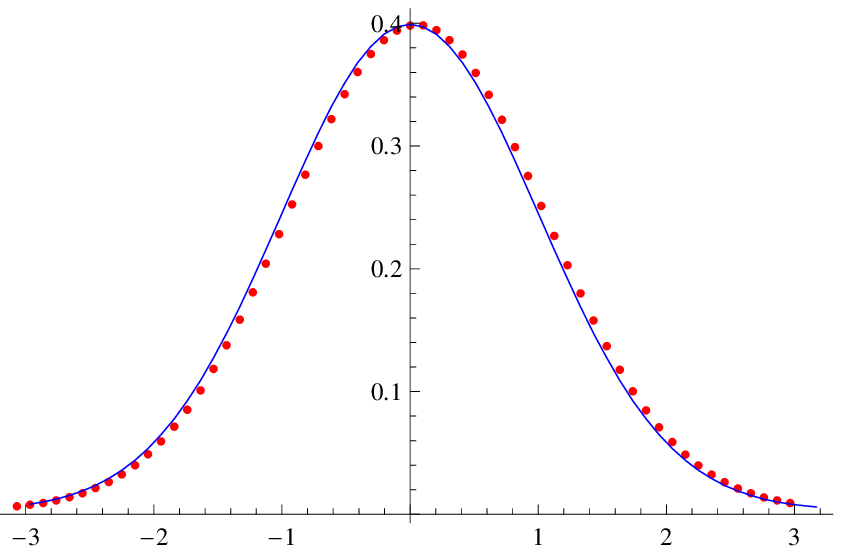}
\end{picture}
\end{minipage}
\hfill
\setlength{\unitlength}{1cm}
\begin{minipage}{5.5cm}
\begin{picture}(5.5,4.0)
\epsfxsize=5.5cm\epsfysize=4cm\epsfbox{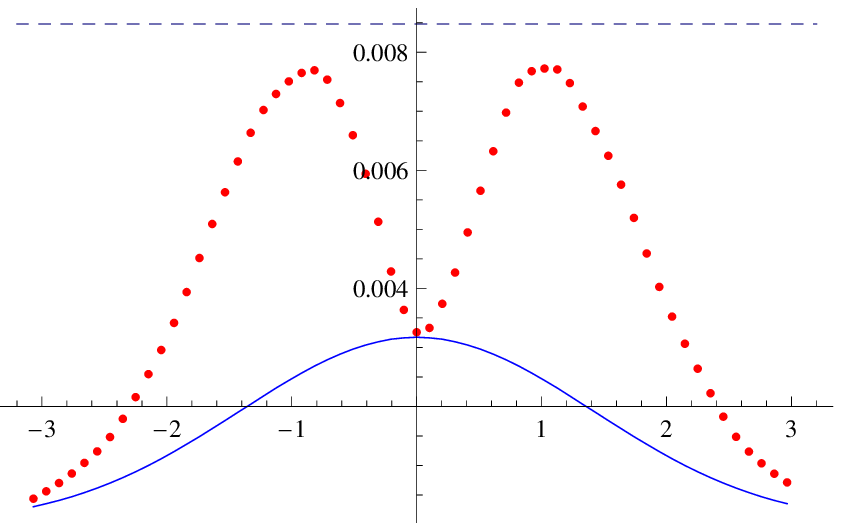}
\end{picture}
\end{minipage}
\caption{\label{add1} \small{The sample means and the theoretical means (left display, a dotted and a solid line, respectively) together with the sample standard deviations and the
two theoretical standard deviations corresponding to Theorems
\ref{thman1} and \ref{thman3} (right display, a dotted, a solid
and a dashed line, respectively). Here the target density $f$ and the density $k$ of a random variable $Z$ are standard normal densities, the noise variance $\sigma_n^2=0.01,$ the sample size $n=100000,$ the bandwidth $h_n=0.05$ and the kernel $w$ is given by \eqref{fankernel}. The number of replications equals $N=500.$ The integral in \eqref{asnrm2} and not its asymptotic expansion was used to evaluate the standard deviation in Theorem
\ref{thman3}.}}
\end{figure}
\begin{figure}[htb]
\setlength{\unitlength}{1cm}
\begin{minipage}{5.5cm}
\begin{picture}(5.5,4.0)
\epsfxsize=5.5cm\epsfysize=4cm\epsfbox{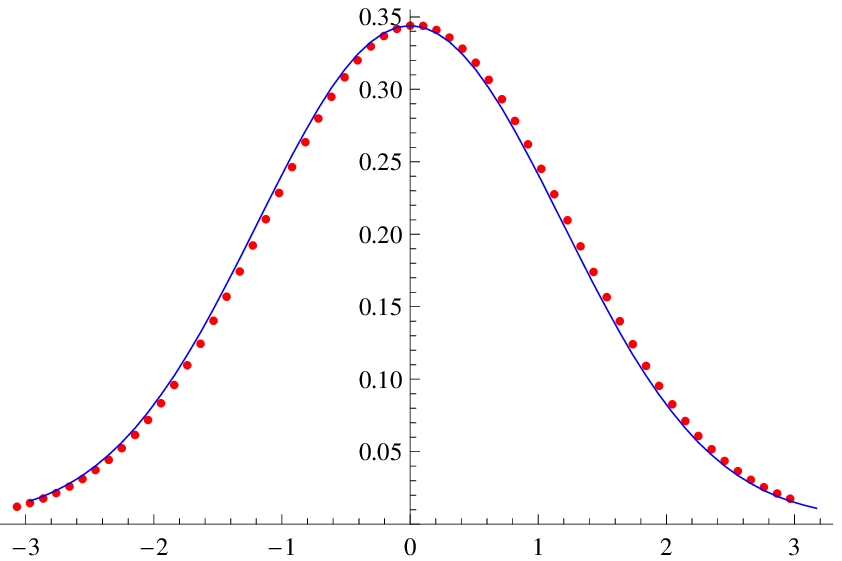}
\end{picture}
\end{minipage}
\hfill
\setlength{\unitlength}{1cm}
\begin{minipage}{5.5cm}
\begin{picture}(5.5,4.0)
\epsfxsize=5.5cm\epsfysize=4cm\epsfbox{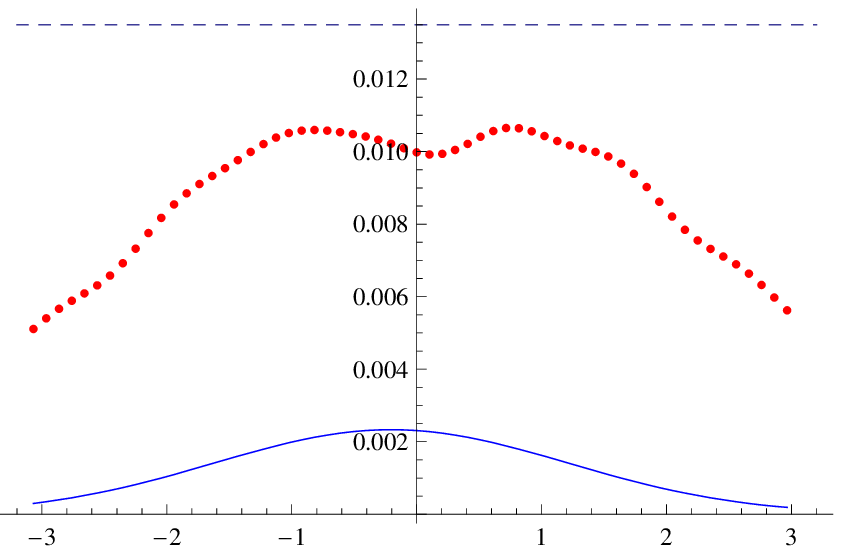}
\end{picture}
\end{minipage}
\caption{\label{add5} \small{The sample means and the theoretical means (left display, a dotted and a solid line, respectively) together with the sample standard deviations and the
two theoretical standard deviations corresponding to Theorems
\ref{thman1} and \ref{thman3} (right display, a dotted, a solid
and a dashed line, respectively). Here the target density $f$ and the density $k$ of a random variable $Z$ are standard normal densities, the noise variance $\sigma_n^2=1,$ the sample size $n=100000,$ the bandwidth $h_n=0.24$ and the kernel $w$ is given by \eqref{fankernel}. The number of replications equals $N=500.$ The integral in \eqref{asnrm2} and not its asymptotic expansion was used to evaluate the standard deviation in Theorem
\ref{thman3}.}}
\end{figure}
\begin{figure}[htb]
\setlength{\unitlength}{1cm}
\begin{minipage}{5.5cm}
\begin{picture}(5.5,4.0)
\epsfxsize=5.5cm\epsfysize=4cm\epsfbox{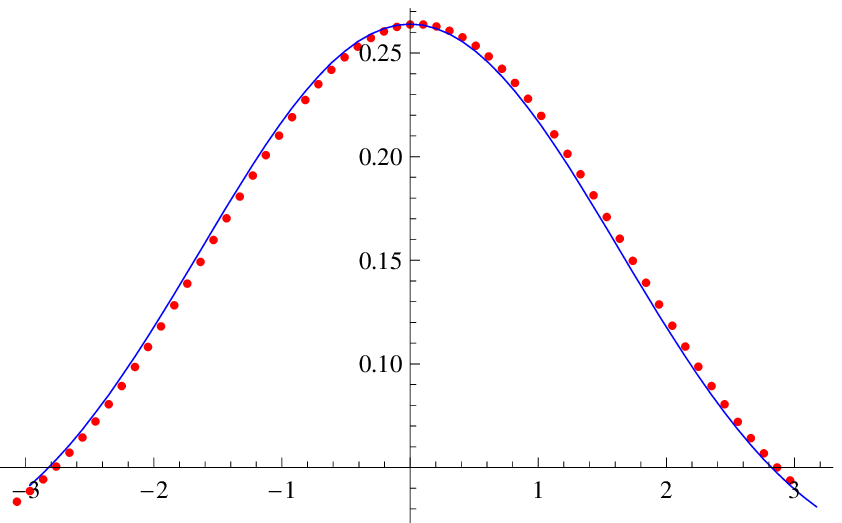}
\end{picture}
\end{minipage}
\hfill
\setlength{\unitlength}{1cm}
\begin{minipage}{5.5cm}
\begin{picture}(5.5,4.0)
\epsfxsize=5.5cm\epsfysize=4cm\epsfbox{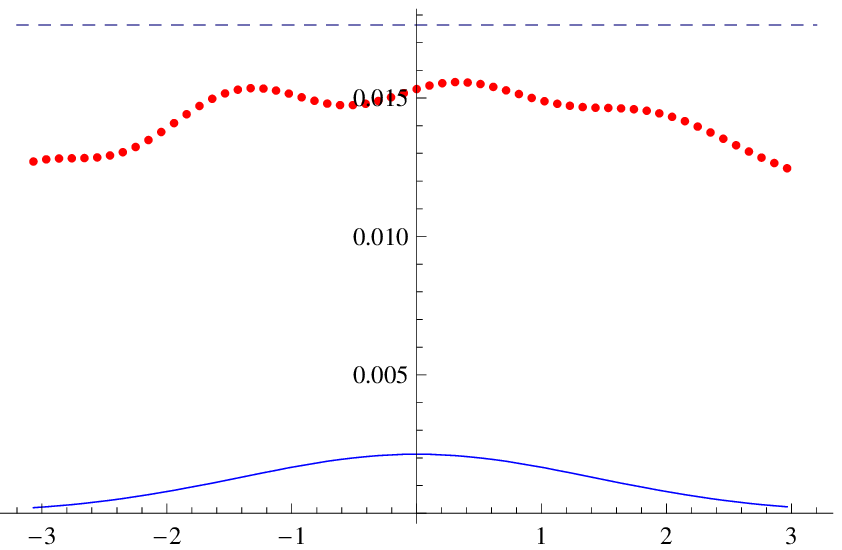}
\end{picture}
\end{minipage}
\caption{\label{add3} \small{The sample means and the theoretical means (left display, a dotted and a solid line, respectively) together with the sample standard deviations and the
two theoretical standard deviations corresponding to Theorems
\ref{thman1} and \ref{thman3} (right display, a dotted, a solid
and a dashed line, respectively). Here the target density $f$ and the density $k$ of a random variable $Z$ are standard normal densities, the noise variance $\sigma_n^2=4,$ the sample size $n=100000,$ the bandwidth $h_n=0.44$ and the kernel $w$ is given by \eqref{fankernel}. The number of replications equals $N=500.$ The integral in \eqref{asnrm2} and not its asymptotic expansion was used to evaluate the standard deviation in Theorem
\ref{thman3}.}}
\end{figure}

In our final example we considered the case when the density $f$ is again a mixture of two normal densities (see above for details). The simulation results for this density are reported in Figure \ref{add2}.
\begin{figure}[htb]
\setlength{\unitlength}{1cm}
\begin{minipage}{5.5cm}
\begin{picture}(5.5,4.0)
\epsfxsize=5.5cm\epsfysize=4cm\epsfbox{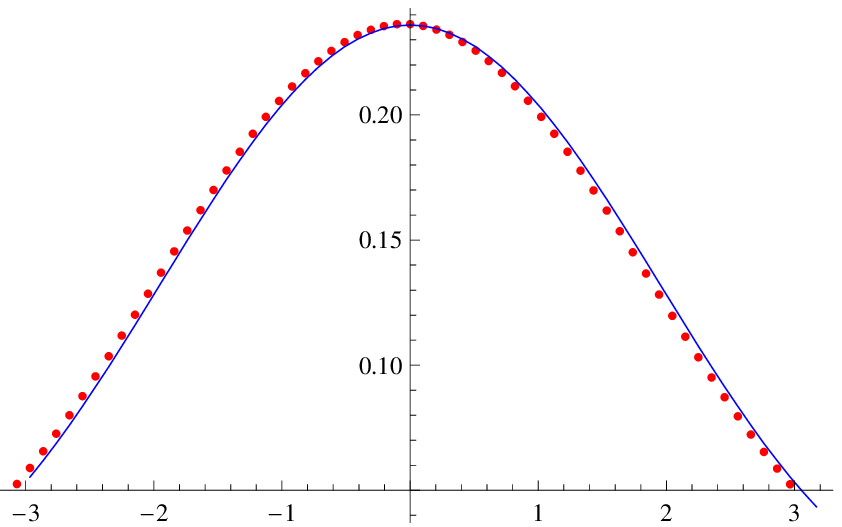}
\end{picture}
\end{minipage}
\hfill
\setlength{\unitlength}{1cm}
\begin{minipage}{5.5cm}
\begin{picture}(5.5,4.0)
\epsfxsize=5.5cm\epsfysize=4cm\epsfbox{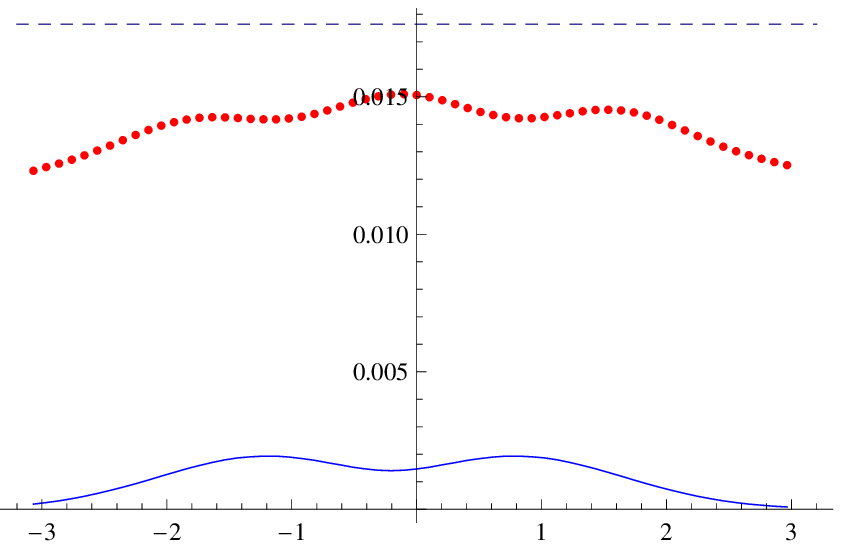}
\end{picture}
\end{minipage}
\caption{\label{add2} \small{The sample means and the
theoretical means (left display, a dotted and a solid line,
respectively) together with the sample standard deviations and the
two theoretical standard deviations corresponding to Theorems
\ref{thman1} and \ref{thman3} (right display, a dotted, a solid
and a dashed line, respectively). Here the target density $f$ is a
mixture of two normal densities with means equal to $-1$ and $1$ and the same variance $0.375,$ the mixing probability is $0.5,$ the density $k$ of a random variable $Z$ is a standard normal density, the noise variance $\sigma_n^2=4,$ the
sample size $n=100000,$ the bandwidth $h_n=0.44$ and the kernel
$w$ is given by \eqref{fankernel}. The number of replications
equals $N=500.$ The integral in \eqref{asnrm2} and not its
asymptotic expansion was used to evaluate the standard deviation
in Theorem \ref{thman3}.}}
\end{figure}
In this last example the bandwidth $h_n=0.44$ was on purpose not selected as a minimiser of $\operatorname{MISE}[f_{nh_n}],$ but was taken to be the same as when estimating a standard normal density (see Figure \ref{add3} above). Notice that the sample standard deviation is almost constant in the neighbourhood of the origin and is of the same magnitude as the one depicted in Figure \ref{add3}. This seems to provide an additional confirmation of the statement of Theorem \ref{thman3}, which says that the limit variance of the estimator $f_{nh_n}$ does not depend on the target density $f.$ Also notice that because of the fact that $h_n$ is relatively large, the smoothed version of $f,$ i.e.\ $f\ast w_{h_n},$ is unimodal instead of being bimodal.

As a preliminary conclusion (we also considered some other examples not reported here), our simulation examples seem to suggest that the asymptotics given by Theorem \ref{thman3} correspond to the less realistic scenarios of high noise level and very large sample size. This provides further motivation for the study of deconvolution problems under the assumption $\sigma_n\rightarrow 0$ as $n\rightarrow\infty.$
\section{Proofs}
\label{proofs}

To prove Theorem \ref{thman1}, we will need the following
modification of Bochner's lemma, see \citet{parzen} for the latter.

\begin{lemma}
\label{fanlemma} Suppose that for all $y$ we have $K_n(y)\rightarrow K(y)$ as $n\rightarrow\infty$ and that $\sup_{n}|K_n(y)|\leq K^{*}(y),$ where the function $K^{*}$ is such that
$\infint K^{*}(y)dy<\infty$ and $\lim_{y\rightarrow\infty}yK^{*}(y)=0.$ Furthermore, suppose that $g_n$ is a sequence of densities, such that
\begin{equation}
\label{suplim}
\lim_{n\rightarrow\infty}\sup_{|u|\leq\epsilon_n}|g_{n}(x-u)-f(x)|\rightarrow
0
\end{equation}
for some sequence $\epsilon_n\downarrow 0,$ such that
$\epsilon_n/h_n\rightarrow\infty$ as $n\rightarrow\infty$ for a sequence $h_n\rightarrow 0.$ Then
\begin{equation}
\label{flemma}
\lim_{n\rightarrow\infty}\frac{1}{h_n}\infint K_n\left(\frac{x-y}{h_n}\right)g_n(y)dy=f(x)\infint K(y)dy.
\end{equation}
\end{lemma}

\begin{proof}
The proof follows the same lines as the proof of Lemma 2.1 in \citet{fan2}. We have
\begin{multline*}
\left|\frac{1}{h_n}\infint K_n\left(\frac{x-y}{h_n}\right)g_n(y)dy-f(x)\infint K(y)dy\right|\\
\leq \left|\frac{1}{h_n}\infint K_n\left(\frac{x-y}{h_n}\right)g_n(y)dy-f(x)\frac{1}{h_n}\infint K_n\left(\frac{y}{h_n}\right) dy\right|\\
+f(x)\left|\infint K_n(y)dy-\infint K(y)dy\right|=I+II.
\end{multline*}
Notice that $II$ converges to zero by the dominated convergence
theorem. We turn to $I.$ Splitting the integration region into the
sets $\{|u|\leq\epsilon_n\}$ and $\{|u|>\epsilon_n\}$ for some
$\epsilon_n>0,$ we obtain that
\begin{align*}
I &\leq \left|\int_{\{|u|\leq\epsilon_n\}}(g_n(x-u)-f(x))\frac{1}{h_n}K_n\left(\frac{u}{h_n}\right)du\right|\\
& +
\left|\int_{\{|u|>\epsilon_n\}}(g_n(x-u)-f(x))\frac{1}{h_n}K_n\left(\frac{u}{h_n}\right)du\right|\\
& =III+IV.
\end{align*}
For $III$ we have
\begin{equation*}
III\leq \sup_{|u|\leq\epsilon_n}|g_n(x-u)-f(x)|\infint K^{*}(u)du.
\end{equation*}
By \eqref{suplim} the right-hand side of the above expression
vanishes as $n\rightarrow\infty.$ Now we consider $IV.$ Using the
fact that $g_n$ is a density (and hence that it is positive and
integrates to one), we have
\begin{align*}
IV &\leq
\int_{|u|>\epsilon_n}g_n(x-u)\frac{1}{h_n}\left|K^{*}\left(\frac{u}{h_n}\right)\right|du+f(x)\int_{|u|>\epsilon_n} \frac{1}{h_n}K^{*}\left(\frac{u}{h_n}\right)du\\
&\leq \frac{1}{\epsilon}\sup_{|y|>\epsilon_n/h_n}
|yK^{*}(y)|+f(x)\int_{|y|>\epsilon_n/h_n}K^{*}(y)dy.
\end{align*}
Notice that the right-hand side in the last inequality vanishes as
$n\rightarrow\infty,$ because we assumed that
$\epsilon_n/h_n\rightarrow\infty.$ Combination of these results
yields the statement of the lemma.
\end{proof}

\begin{proof}[Proof of Theorem \ref{thman1}] The main steps of the proof are similar to those on pp.\ 1069--1070 of \citet{parzen}. Let $\delta$ be an arbitrary positive number. Denote
\begin{equation*}
V_{nj}=\frac{1}{h_n}w_{r_n}\left(\frac{x-X_j}{h_n}\right),
\end{equation*}
where $w_{r_n}$ is defined by \eqref{wh} and notice that \eqref{fnh2} is an average of the i.i.d.\ random variables $V_{n1},\ldots,V_{nn}.$ We have
\begin{equation}
\label{an1}
\var[V_{nj}]=\ex[V_{nj}^2]-(\ex[V_{nj}])^2.
\end{equation}
Observe that
\begin{equation}
\label{exvnjsq} \ex[V_{nj}^2]=\infint
\frac{1}{h_n^2}\left|w_{r_n}\left(\frac{x-y}{h_n}\right)\right|^2g_n(y)dy,
\end{equation}
where $g_n$ denotes the density of $X_j.$ Integration by parts
gives
\begin{equation*}
w_{r_n}(u)=\frac{1}{iu}\int_{-1}^{1}e^{-itu}\left(\frac{\phi_w(t)}{\phi_k(r_nt)}\right)^{'}dt,
\end{equation*}
and hence
\begin{equation*}
|w_{r_n}(u)|\leq\frac{1}{|u|}\int_{-1}^{1}\left|\frac{\phi_w^{'}(t)\phi_k(r_nt)-r_n\phi_w(t)\phi_k^{'}(r_nt)}{(\phi_k(r_nt))^2}\right|dt
\end{equation*}
Furthermore, $\lim_{n\rightarrow\infty} r_n=r<\infty$ implies that there exists
a positive number $a,$ such that $\sup r_n\leq a<\infty.$ Notice
that
\begin{equation*}
\inf_{t\in[-1,1]}|\phi_k(r_nt)|=\inf_{s\in[-r_n,r_n]}|\phi_k(s)|\geq\inf_{s\in[-a,a]}|\phi_k(s)|.
\end{equation*}
Therefore
\begin{equation}
\label{wr1}
|w_{r_n}(u)|\leq c_{ak}\frac{1}{|u|}\int_{-1}^{1}(|\phi_w^{'}(t)|+|\phi_w(t)|)dt,
\end{equation}
where the constant $c_{ak}$ does not depend on $n,$ but only on
the density $k$ and the number $a.$ On the other hand
\begin{equation}
\label{wr2} |w_{r_n}(u)|\leq
\frac{1}{2\pi}\int_{-1}^{1}\frac{|\phi_w(t)|}{\inf_{s\in[-a,a]}|\phi_k(s)|}dt<\infty.
\end{equation}
Combining \eqref{wr1} and \eqref{wr2}, we obtain that
\begin{equation}
\label{wr3}
|w_{r_n}(u)|\leq\min\left(C_1,\frac{C_2}{|u|}\right),
\end{equation}
where the constants $C_1$ and $C_2$ do not depend on $n.$ Observe
that the function on the right-hand side of \eqref{wr3} is square
integrable. Next, we have
\begin{equation*}
\sup_{|u|\leq\epsilon_n}|g_n(x-u)-f(x)|\leq
\sup_{|u|\leq\epsilon_n}|g_n(x-u)-g_n(x)|+|g_n(x)-f(x)|=I+II.
\end{equation*}
for an arbitrary $\epsilon_n>0.$ By the Fourier inversion argument
for $I$ we obtain
\begin{align*}
|I|&\leq \left|\sup_{|u|\leq\epsilon_n}\frac{1}{2\pi}\infint
e^{-itx}\phi_f(t)\phi_k(r_n t)(e^{itu}-1)dt \right|\\
&\leq \frac{1}{2\pi}\infint
|\phi_f(t)|\sup_{|u|\leq\epsilon_n}|e^{itu}-1|dt.
\end{align*}
Notice that $\sup_{|u|\leq\epsilon_n}|e^{itu}-1|\leq \epsilon_n
|t|\rightarrow 0$ for every fixed $t.$ Furthermore,
$\sup_{|u|\leq\epsilon_n}|e^{itu}-1|\leq 2$ and $\phi_f$ is
integrable. Let $\epsilon_n\downarrow 0$ as $n\rightarrow \infty.$
Then by the dominated convergence theorem $I$ will vanish as
$n\rightarrow \infty.$ A similar Fourier inversion argument and
another application of the dominated convergence theorem shows
that $II$ also vanishes as $n\rightarrow\infty.$ Thus
\eqref{suplim} is satisfied. Now \eqref{exvnjsq}, \eqref{wr3} and
Lemma \ref{fanlemma} imply that
\begin{equation}
\label{an2} \ex[V_{nj}^2]\sim \frac{1}{h_n}f(x)\infint|w_r(u)|^2du.
\end{equation}
Furthermore, by Fubini's theorem
\begin{equation}
\begin{split}
\label{an3}
\ex[V_{nj}]&=\frac{1}{h_n}\frac{1}{2\pi}\infint \exp\left(\frac{-i{t}x}{h_n}\right)\ex\left[\exp\left(\frac{i{t}X_j}{h_n}\right)\right]
\frac{\phi_w(t)}{\phi_k(r_nt)}dt\\
&=\infint
\exp\left(-\frac{i{t}x}{h_n}\right)\ex\left[\exp\left(\frac{i{t}Y_j}{h_n}\right)\right]\ex\left[\exp\left(\frac{i{t}\sigma_n
Z_j}{h_n}\right)\right]
\frac{\phi_w(t)}{\phi_k(r_nt)}dt\\
&=\frac{1}{h_n}\frac{1}{2\pi}\infint \exp\left(-\frac{i{t}x}{h_n}\right)\phi_f\left(\frac{t}{h_n}\right)\phi_w(t)dt.
\end{split}
\end{equation}
The last expression is bounded uniformly in $h_n$ due to
Assumption A (i) and (iii), which can be seen by a change of the
integration variable $t/h_n=s.$ Moreover, using \eqref{exvnjsq},
\eqref{wr2} and \eqref{an2}, we have that
\begin{equation}
\label{an4} \ex[|V_{nj}^{2+\delta}|]=\infint
\frac{1}{h^{2+\delta}}\left|w_{r_n}\left(\frac{x-y}{h}\right)\right|^{2+\delta}g_n(y)dy
\end{equation}
is of order $h_n^{-1-\delta}.$ Combination of the above results now yields
\begin{equation}
\frac{\ex[|V_{nj}-\ex[V_{nj}]|^{2+\delta}]}{n^{\delta/2}(\var[V_{nj}])^{1+\delta/2}}\rightarrow 0
\end{equation}
as $h_n\rightarrow 0,nh_n\rightarrow\infty.$ Therefore
$f_{nh_n}(x)$ satisfies Lyapunov's condition for asymptotic
normality in the triangular array scheme, see Theorem 7.3 in
\citet{billingsley}, and hence it is asymptotically normal, i.e.\
\begin{equation*}
\frac{f_{nh_n}(x)-\ex[f_{nh_n}(x)]}{\sqrt{\var[f_{nh_n}(x)]}}\convd {\mathcal N}(0\, , 1).
\end{equation*}
Formula \eqref{an0} is then immediate from this fact, formulae
\eqref{an1}, \eqref{an2}, \eqref{an3} and Slutsky's lemma, see
Corollary 2 on p.~31 of \citet{billingsley}.
\end{proof}

\begin{proof}[Proof of Theorem \ref{thman3}] The proof follows the same line of thought as the proof of Theorem 1 in \citet{vanes1}. For an arbitrary $0<\epsilon<1$ we have
\begin{align}
f_{nh_n}(x)
&=
\frac{1}{2\pi nh_n}\sum_{j=1}^n\int_{-\epsilon}^\epsilon
\exp\Big(is\Big(\frac{X_j-x}{h_n}\Big)\Big)\frac{\phi_w(s)}{\phi_k(s/\rho_n)}ds
\label{epsint}\\
&+
\frac{1}{2\pi nh_n}\sum_{j=1}^n
\Big(
\int_{-1}^{-\epsilon}+\int_{\epsilon}^1\Big)
\exp\Big(is\Big(\frac{X_j-x}{h_n}\Big)\Big)\frac{\phi_w(s)}{\phi_k(s/\rho_n)}ds
\label{1int}.
\end{align}
The integral in \eqref{epsint} is real-valued, which can be seen
by taking its complex conjugate. Using Assumption B (i), the
variance of \eqref{epsint} can be bounded as follows:
\begin{align*}
\var  &\left[ \frac{1}{2\pi nh_n}\sum_{j=1}^n \int_{-\epsilon}^\epsilon
\exp\left(is\left(\frac{X_j-x}{h_n}\right)\right)\frac{\phi_w(s)}{\phi_k(s/\rho_n)}ds \right] \\
& \leq
\frac{1}{4\pi^2 nh_n^2}\ex \left[\left( \int_{-\epsilon}^\epsilon
\exp\Big(is\Big(\frac{X_j-x}{h_n}\Big)\Big)\frac{\phi_w(s)}{\phi_k(s/\rho_n)}ds
\right)^2\right]\\
&\leq \frac{1}{4\pi^2 nh_n^2 }\left( \int_{-\epsilon}^\epsilon
\frac{1}{|\phi_k(s/\rho_n)|}ds \right)^2 \\
& \leq \frac{1}{4\pi^2 nh_n^2 }
(2\epsilon)^2 \left(\frac{1}{\inf_{-\epsilon\leq s\leq \epsilon}|\phi_k(s/\rho_n)|}
\right)^2\\
&=O\left( \frac{1}{\pi^2 }\frac{1}{n}\frac{1}{\sigma_n^2}\,
\left(\frac{\epsilon}{\rho_n}\right)^{2-2\lambda_0} \exp\left(
\frac{2\epsilon^{\lambda}}{\mu\rho_n^{\lambda}}\right)\right).
\end{align*}
Hence the contribution of (\ref{epsint}) minus its expectation is
of order
$$
O_P\left(\frac{1}{\sigma_n}\frac{1}{\sqrt{n}}\left(\frac{\epsilon}{\rho_n}\right)^{1-\lambda_0} \exp\left(
\frac{\epsilon^{\lambda}}{\mu\rho_n^{\lambda}}\right)\right).
$$
By comparing this to the normalising constant in
\eqref{supersmooth}, by Slutsky's lemma we see that \eqref{epsint}
can be neglected when considering the asymptotic normality of
$f_{nh_n}(x).$

The term \eqref{1int} can be written as
\begin{align}
&\frac{1}{2\pi nh_nC}\sum_{j=1}^n
\Big(\int_{-1}^{-\epsilon}+\int_{\epsilon}^1\Big)
\exp\Big(is\Big(\frac{X_j-x}{h_n}\Big)\Big)\phi_w(s)\left(\frac{|s|}{\rho_n}\right)^{-\lambda_0}
\exp\left(\frac{|s|^{\lambda}}{\mu \rho_n^{\lambda}}\right)ds\label{versc}\\
&+
\frac{1}{2\pi nh_n}\sum_{j=1}^n
\Big(\int_{-1}^{-\epsilon}+\int_{\epsilon}^1\Big)
\exp\Big(is\Big(\frac{X_j-x}{h_n}\Big)\Big)\phi_w(s)\nonumber\\
&\times\Big(\frac{1}{\phi_k(s/\rho_n)}-\frac{1}{C}\Big(\frac{|s|}{\rho_n}\Big)^{-\lambda_0}
\exp\left(\frac{|s|^{\lambda}}{\mu \rho_n^{\lambda}}\right)\Big)
ds.\label{verschil}
\end{align}
Observe that both \eqref{versc} and \eqref{verschil} are real. Expression \eqref{versc} equals
\begin{equation}
\label{realint} \frac{1}{\pi n\sigma_n
C}\,\rho_n^{\lambda_0-1}\sum_{j=1}^n \int_\epsilon^1
\cos\Big(s\Big(\frac{X_j-x}{h_n}\Big)\Big)\phi_w(s) s^{-\lambda_0}
\exp\left(\frac{s^{\lambda}}{\mu \rho_n^{\lambda}}\right)ds.
\end{equation}
By formula (21) of \citet{vanes1}
\begin{equation}
\cos\Big(s\Big(\frac{X_j-x}{h_n}\Big)\Big)=
\cos\Big(\frac{X_j-x}{h_n}\Big) +R_{n,j}(s),\label{trig1}
\end{equation}
where $R_{n,j}(s)$ is a remainder term satisfying
\begin{equation}\label{Rnbound1}
|R_{n,j}|\leq ( |x| +|X_j|)\Big(\frac{1-s}{h_n}\Big),
\end{equation}
whence by Lemma 5 of \citet{vanes1} the expression \eqref{realint}
equals
\begin{align*}
&\frac{1}{\pi\sigma_n C}\rho_n^{\lambda_0-1}\int_\epsilon^1 \phi_w(s)
s^{-\lambda_0} \exp\left(\frac{s^\lambda}{\mu \rho_n^\lambda}\right)ds \frac{1}{n}\sum_{j=1}^n \cos\Big(\frac{X_j-x}{h_n}\Big)+
\frac{1}{n}\sum_{j=1}^n \tilde{R}_{n,j}\\
&= \frac{1}{\pi\sigma_n C}A(\Gamma(\alpha+1)+o(1))\left(\frac{\mu}{\lambda}\right)^{1+\alpha}
\rho_n^{\lambda(1+\alpha)+\lambda_0-1}\zeta(\rho_n) \frac{1}{n}\sum_{j=1}^n \cos\Big(\frac{X_j-x}{h_n}\Big)\\
&+ \frac{1}{n}\sum_{j=1}^n \tilde{R}_{n,j},
\end{align*}
where
$$
\tilde{R}_{n,j}=\frac{1}{\pi\sigma_n C}\rho_n^{\lambda_0-1}\int_\epsilon^1
R_{n,j}(s) \phi_w(s)
s^{-\lambda_0} \exp\left(\frac{s^{\lambda}}{\mu \rho_n^\lambda}\right)ds.
$$
By \eqref{Rnbound1} and Lemma 5 of \citet{vanes1} the latter
expression can be bounded as
\begin{eqnarray*}
\lefteqn{
|\tilde{R}_{n,j}|\leq \frac{1}{\pi\sigma_n C}(|x|+|X_j|)\rho_n^{\lambda_0-1}
\int_\epsilon^1
\Big(\frac{1-s}{h_n}\Big)\phi_w(s)
s^{-\lambda_0} \exp\left(\frac{s^{\lambda}}{\mu \rho_n^\lambda}\right)ds}\\
&=&
\frac{1}{\pi\sigma_n h_n C}A\left(\frac{\mu}{\lambda}\right)^{\alpha+2}(\Gamma(\alpha+2)+o(1))
\rho_n^{\lambda(2+\alpha)+\lambda_0-1}\zeta(\rho_n)
(|x|+|X_j|).
\end{eqnarray*}
Hence
\begin{equation*}
\var [\tilde{R}_{n,j}]\leq \ex
[\tilde{R}_{n,j}^2]=O\left(\frac{1}{\sigma_n^2 h_n^2}
\rho_n^{2(\lambda(2+\alpha)+\lambda_0-1)}(\zeta(\rho_n))^2\right).
\end{equation*}
Here we used the fact that $\ex[Y^2_j]+\ex[Z^2_j]<\infty$ together with the
fact that being convergent, the sequence $\sigma_n$ is bounded,
which implies that $\ex[X_j^2]$ is bounded uniformly in $n.$ By
Chebyshev's inequality it follows that
\begin{equation}
\label{rmn}
\frac{1}{n} \sum_{j=1}^n (\tilde{R}_{n,j}-\ex\tilde{R}_{n,j})=
O_P\left(\frac{1}{\sigma_n h_n} \frac{\rho_n^{\lambda(2+\alpha)+\lambda_0-1}\zeta(\rho_n) }{\sqrt{n}}\right).
\end{equation}
After multiplication of this term by the normalising factor from
\eqref{supersmooth} we obtain that the resulting expression is of
order $\rho_n^\lambda(\sigma_n
h_n)^{-1}=h_n^{\lambda-1}\sigma_n^{\lambda}.$ Assumption B (v) and
Slutsky's lemma then imply that the remainder term \eqref{rmn} can
be neglected when considering the asymptotic normality of
$f_{nh_n}(x).$

The variance of \eqref{verschil} can be bounded by
\begin{equation*}
\frac{1}{4\pi^2 nh_n^2C^2}
\Big(\Big( \int_{-1}^{-\epsilon} +\int_{\epsilon}^1\Big)
|\phi_w(s)|
\left(\frac{|s|}{\rho_n}\right)^{-\lambda_0}
\exp\left(\frac{|s|^\lambda}{\mu \rho_n^\lambda}\right)|u(s/\rho_n)|ds\Big)^2,
\end{equation*}
where the function $u$ is given by
\begin{equation}\label{ufunction}
u(y)=
\frac{C|y|^{\lambda_0} \exp(-|y|^\lambda\mu^{-1})}{\phi_k(y)}-1.
\end{equation}
This function is bounded on $\mathbb{R}\negmedspace\setminus
\negmedspace (-\delta,\delta),$ where $\delta$ is an arbitrary
positive number. It follows that $u(s/\rho_n)$ is also bounded and
tends to zero for all fixed $s$ with $|s| \geq \epsilon$ as
$\rho_n \rightarrow 0.$  Hence the variance of \eqref{verschil} is
of smaller order compared to the variance of \eqref{versc}, which
can be shown by the dominated convergence theorem via an argument
similar to the one in the proof of Lemma 5 of \citet{vanes1}.
Therefore by Slutsky's lemma \eqref{verschil} can be neglected
when considering asymptotic normality of \eqref{supersmooth}.

Combination of the above observations yields that it suffices to study
\begin{equation}
\label{decomp}
\frac{A}{\pi C}\left(\frac{\mu}{\lambda}\right)^{1+\alpha}
(\Gamma(\alpha+1)+o(1))U_{nh_n}(x),
\end{equation}
where
\begin{equation*}
U_{nh_n}(x)=\frac{1}{\sqrt{n}}\sum_{j=1}^n \left(\cos\left(\frac{X_j-x}{h_n}\right)
-\ex \left[ \cos\left(\frac{X_j-x}{h_n}\right)\right]\right).
\end{equation*}
Observe that
\begin{equation*}
\frac{X_j-x}{h_n}=\frac{Y_j-x}{h_n}+\frac{\sigma_n}{h_n}Z_j=\frac{Y_j-x}{h_n}+\frac{Z_j}{\rho_n}
\end{equation*}
and that by the same arguments as in the proof of Lemma 6 in
\citet{vanes1}, both $(Y_j-x)/h_n\negthinspace \mod 2\pi$ and
$Z/\rho_n\negthinspace\mod 2\pi$ converge in distribution to a
random variable with a uniform distribution on $[0,2\pi].$
Furthermore, these two random variables are independent. Now
notice that for two independent random variables $W_1$ and $W_2$
the sum $W_1+W_2\negthinspace\mod 2\pi$ equals in distribution
$(W_1\mod 2\negthinspace\pi+W_2\negthinspace\mod
2\pi)\negthinspace\mod 2\pi.$ Moreover, if $W_1$ and $W_2$ are
uniformly distributed on $[0,2\pi],$ then also
$W_1+W_2\negthinspace\mod 2\pi$ is uniformly distributed on
$[0,2\pi],$ see \citet{scheinok}. Using these two facts, by
exactly the same arguments as in the proof of Lemma 6 of
\citet{vanes1} we finally obtain that $U_{nh_n}(x)\convd {\mathcal
N}\left(0,1/2\right).$ The latter in conjunction with
\eqref{decomp} entails \eqref{supersmooth}.
\end{proof}

\begin{proof}[Proof of Theorem \ref{thman2}]
The proof employs an approach similar to the proof of Theorem 2.1 of \citet{fan2}. We have
\begin{equation*}
\ex[V_{nj}^2]=\infint
\frac{1}{h_n^2}\left|w_{{\rho}_n}\left(\frac{x-y}{h_n}\right)\right|^2g_n(y)dy.
\end{equation*}
By equation (3.1) of \citet{fan2} (with $h_n$ replaced by $\rho_n$) we have
\begin{equation*}
\left|\frac{\rho_n^{\beta}\phi_w(t)}{\phi_k(t/\rho_n)}\right|\leq w_0(t),
\end{equation*}
where $w_0$ is a positive integrable function. Hence by the dominated convergence theorem
\begin{equation*}
\rho_n^{\beta}w_{\rho_n}(y)\rightarrow \frac{1}{2\pi C}\int_{-1}^{1}e^{-itx}t^{\beta}\phi_w(t)dt.
\end{equation*}
Furthermore, again by equation (3.1) of \citet{fan2} we have $|\rho_n^{\beta}w_{\rho_n}(y)|\leq C_2$ for some constant $C_2$ independent of $n$ and $y,$ while equation (2.7) of \citet{fan2} implies that $|\rho_n^{\beta}w_{\rho_n}(y)|\leq C_1/|y|.$ Combination of these two bounds gives
\begin{equation}
\label{rhob1}
|\rho_n^{\beta}w_{\rho_n}(y)|\leq \min\left(\frac{C_1}{|y|},C_2\right).
\end{equation}
Since the fact that $g_n$ satisfies \eqref{suplim} can be shown
exactly as in the proof of Theorem \ref{thman1}, by Lemma
\ref{fanlemma} we then obtain that
\begin{equation}
\label{an13}
\begin{split}
\ex[V_{nj}^2]
&\sim \frac{f(x)}{h_n\rho_n^{2\beta}}\infint\left[\frac{1}{2\pi C}\int_{-1}^{1}e^{-ity}t^{\beta}\phi_w(t)dt\right]^2dy\\
&=\frac{1}{h_n\rho_n^{2\beta}}\frac{f(x)}{2\pi C^2}\int_{-1}^{1}|t|^{2\beta}|\phi_w(t)|^2dt,
\end{split}
\end{equation}
where the last equality follows from Parseval's identity.
Furthermore, by Fubini's theorem and the dominated convergence
theorem we have
\begin{equation}
\label{an14}
\begin{split}
\ex[V_{nj}]&=\frac{1}{h_n}\frac{1}{2\pi}\infint
\exp\left(-\frac{i{t}x}{h_n}\right)\ex\left[\exp\left(\frac{i{t}X_j}{h_n}\right)\right]
\frac{\phi_w(t)}{\phi_k(t/\rho_n)}dt\\
&=\frac{1}{h_n}\frac{1}{2\pi}\infint e^{-i{t}x/h}\phi_f\left(\frac{t}{h_n}\right)\phi_w(t)dt\\
&=\frac{1}{2\pi}\infint e^{-i{t}x}\phi_f({t})\phi_w(h_nt)dt\\
&\rightarrow f(x).
\end{split}
\end{equation}
The dominated convergence theorem is applicable because of
Assumption B (i) and (iii). Finally, let us consider $\ex[|V_{nj}^{2+\delta}|].$ Writing
\begin{equation}
\label{an15} \ex[|V_{nj}|^{2+\delta}]=\infint
\frac{1}{h_n^{2+\delta}}\left|w_{{\rho}_n}\left(\frac{x-y}{h_n}\right)\right|^{2+\delta}g_n(y)dy,
\end{equation}
and using \eqref{rhob1} and Lemma 1, we obtain that
\begin{equation*}
\ex[|V_{nj}|^{2+\delta}]=O(h_n^{-1-\delta}\rho_n^{-\beta(2+\delta)}).
\end{equation*}
Combination of \eqref{an13}, \eqref{an14} and \eqref{an15} yields
that Lyapunov's condition is fulfilled and hence that
$f_{nh_n}(x)$ is asymptotically normal. Formula
\eqref{ordinarysmooth} then follows from \eqref{an13} and
\eqref{an14}. This completes the proof.
\end{proof}

\end{document}